\documentclass[11pt]{amsart}
 
 \usepackage{graphicx}
 \usepackage{amssymb}
 \usepackage{amsmath}
 \usepackage{amsthm}
 \usepackage{mathtools}
 \usepackage{tikz}
 \usepackage{psfrag}
 \usepackage{xcolor}

\newtheorem{thm}{Theorem}[section]
 
\newtheorem{lem}[thm]{Lemma}
\newtheorem{pro}[thm]{Proposition}
\newtheorem{cor}[thm]{Corollary}

\theoremstyle{definition}
\newtheorem{defn}{Definition}[section] 
\newtheorem{remk}{Remark}[section]

\newtheorem{example}{Example}[section]

\newcommand{\M}{\mathbb M}

\newcommand{\R}{\mathbb R}

\newcommand{\K}{\mathcal K}

\newcommand{\tr}{\operatorname{tr}}

\newcommand{\supp}{\operatorname{supp}}
\newcommand{\dv}{\operatorname{div}}

\newcommand{\cof}{\operatorname{cof}}

\newcommand{\aint}{ {\int \hspace{-11pt} - \hspace{-1pt}}} 
 
 \makeatletter
\numberwithin{equation}{section}
\makeatother

\begin{document}

\title{On a class of special Euler-Lagrange equations}   

 \author{Baisheng Yan}
 \address{Department of Mathematics, Michigan State University, East Lansing, MI 48824, USA}    \email{yanb@msu.edu}

 \begin{abstract}
We make some remarks  on the  Euler-Lagrange equation of energy functional $I(u)=\int_\Omega f(\det Du)\,dx,$ where $f\in C^1(\R).$ For  certain weak solutions $u$ we show that the function $f'(\det Du)$ must  be a  constant  over the domain $\Omega$ and thus, when $f$ is convex, all such solutions are an energy minimizer of $I(u).$  However, other weak solutions exist such that $f'(\det Du)$ is not constant on $\Omega.$ We also prove some results concerning the   homeomorphism solutions, non-quasimonotonicty, radial solutions, and some special properties and questions in the  2-D cases.
   \end{abstract}
 


\maketitle

\section{Introduction}
Let $n\ge 2$ be an integer. We denote by $\M^{n\times n}$  the standard space of real $n\times  n$ matrices with inner product 
\[
A:B =\sum_{i,j=1}^n A^i_j B^i_j \qquad \forall\; A=(A^i_j),\; \; B=(B^i_j)\in \M^{n\times n}.
\]
 Also denote by $A^T$, $\det A,$ and  $\cof A$ the transpose, determinant, and the cofactor matrix of  $A\in \M^{n\times n},$ respectively,  so that the identity $A^T\cof A=(\cof A)A^T=(\det A)I$ holds for all $A\in \M^{n\times n}.$ 
 
Let $ f\colon \R\to\R$ be a $C^1$ function; i.e., $f\in C^1(\R),$  and $\Omega$ be a bounded domain in $\R^n.$  We study the special  energy functional
 \begin{equation}\label{energy}
 I_\Omega(u)=\int_\Omega f(\det Du)dx
 \end{equation}
 for  functions $u\colon \Omega\to\R^n,$ where   $Du$ is the Jacobian matrix of $u$; that is, $(Du)^i_j=\partial u^i/\partial x_j$ if $u =(u^1, \dots,u^n)$ and $x=(x_1, \dots,x_n)\in\Omega.$   Such a functional  $I_\Omega$  and its Euler-Lagrange equations have been studied in \cite{Da81,ESG05, KMS, Ta02} concerning the energy minimizers,  weak solutions and the  gradient flows.  Our motivation to further study the functional $I_\Omega$  is closely related to the counterexamples on existence and regularity  given in  \cite{MSv2, Sz04, Ya1}. From  the identity
\[
 \partial \det A/\partial A^i_j=(\cof A)^i_j; \;\; \mbox{i.e.,} \;\;  D(\det A) = \cof A \quad \forall\, A=(A^i_j) \in \M^{n\times n},
\]
it follows that the Euler-Lagrange equation of   $I_\Omega$ is given by
 \begin{equation}\label{EL-f}
 \dv (f'(\det Du)\cof Du)=0 \quad\mbox{in $\Omega.$} 
 \end{equation}

By a  weak solution of equation (\ref{EL-f}) we mean  a function $u\in W^{1,1}_{loc}(\Omega;\R^n)$ such that $f'(\det Du)\cof Du\in L^1_{loc}(\Omega;\M^{n\times n})$ and
\begin{equation} \label{weak}
\int_\Omega f'(\det Du)\cof Du : D\zeta \,dx =0 \quad \forall\, \zeta\in C_c^1(\Omega;\R^n).
\end{equation} 
This is a very weak sense of solutions.  Usually, under suitable growth conditions on $f,$ the weak solutions are  studied in the space $u\in W^{1,p}_{loc}(\Omega;\R^n)$ for some $1\le p\le \infty$ with  $f'(\det Du)\cof Du\in L^{p'}_{loc}(\Omega;\M^{n\times n})$, where $p'= {p}/{(p-1)}$ is the H\"older conjugate of $p.$  However, in this paper, we will not specify the growth conditions on $f,$ but may assume some  higher integrability of $f'(\det Du)\cof Du.$ For example, if a weak solution $u$ satisfies, in addition, $f'(\det Du)\cof Du\in L^q(\Omega;\M^{n\times n})$ for some $q\ge 1$, then (\ref{weak}) holds for all $\zeta\in W^{1,q'}_0(\Omega;\R^n).$ 

It is well-known that $\dv(\cof Du)=0$    for all $u\in C^2(\Omega;\R^n)$ and thus by approximation   it holds that
\begin{equation} \label{weak-0}
\int_\Omega  \cof Du : D\zeta \,dx =0 \quad \forall\, u\in W^{1,n-1}_{loc}(\Omega;\R^n),\; \zeta\in C_c^1(\Omega;\R^n).
\end{equation} 
Consequently, any functions $u\in W^{1,n-1}_{loc}(\Omega;\R^n)$ with $f'(\det Du)$ being constant almost everywhere  on $\Omega$ are automatically a weak solution of equation (\ref{EL-f}); however, the converse  is false, at least when $n=2$ (see Example \ref{ex1} below).  Some formal calculations based on the identity $\dv(\cof Du)=0$ show  that $f'(\det Du)$ must be constant on $\Omega$ if  $u\in C^2(\Omega;\R^n)$ is a classical  solution of  (\ref{EL-f}) (see \cite{ESG05,Ta02}). We prove that  the same  result holds  for all weak solutions $u\in C^1(\Omega;\R^n)$   (see also  \cite{KMS}) or  $u\in W^{1,p}_{loc}(\Omega;\R^n)$ with $f'(\det Du)\in W_{loc}^{1,\bar p}(\Omega)$ for some $p\ge n-1$ and $\bar p=p/(p+1-n).$ 
If, in addition, $f$ is convex,  then all  weak solutions $u\in W^{1,n}_{loc}(\Omega;\R^n)$  of  (\ref{EL-f}) with   $f'(\det Du)$ being constant almost everywhere on $\Omega$ must  be an energy minimizer of the energy $I_G$ among all $W^{1,n}(G;\R^n)$ functions having  the same boundary value as $u$ on $\partial G$ for all subdomains $G$ of $\Omega.$ 

We study further conditions under which a weak solution  $u$ of  (\ref{EL-f})  must have constant $f'(\det Du)$ almost everywhere on $\Omega.$ In particular, we prove that this is the case for homeomorphism weak solutions $u\in W^{1,p}(\Omega;\R^n)$ for $p>n$ (see Theorem \ref{thm-homeo}).  When the domain $\Omega$ is a ball, we study the  radially symmetric weak solutions and establish a specific characterization about certain weak solutions (see Section \ref{sec-radial}). In the 2-D cases, the equation (\ref{EL-f}) is reformulated as a first-order partial differential relation \cite{KMS,MSv2} and we prove some interesting results  for such a relation  (see Section \ref{sec-2-D}).

\section{Some general results for all dimensions}
In what follows, unless otherwise specified, we assume   $h\in C(\R)$ is a given function. Recall that  $u\in W^{1,1}_{loc}(\Omega;\R^n)$ is a weak solution of  equation
\begin{equation}\label{EL-h}
 \dv (h(\det Du)\cof Du)=0 \quad\mbox{in $\Omega,$} 
 \end{equation}
 provided that  $h(\det Du)\cof Du\in L^1_{loc}(\Omega;\M^{n\times n})$ and
\begin{equation} \label{weak-1}
\int_\Omega h(\det Du)\cof Du : D\zeta \,dx =0 \quad \forall\, \zeta\in C_c^1(\Omega;\R^n).
\end{equation}
If, in addition, $h(\det Du)\cof Du\in L_{loc}^q(\Omega;\M^{n\times n})$ for some $q\ge 1,$ then equation (\ref{weak-1}) holds for all $\zeta\in W^{1,q'}(\Omega;\R^n)$ with compact support in $\Omega,$  where $q'=\frac{q}{q-1}.$
 
\subsection{Formal calculations and some  implications} It is well-known that   $\dv(\cof Dw)=0
$ holds in $\Omega$ for all $w\in C^2(\Omega;\R^n).$  From this, we easily have  the following point-wise identity:
\[
\dv(a \cof Dw )=(\cof Dw) Da \quad \forall\,a\in C^1(\Omega),\;w\in C^2(\Omega;\R^n).
\]
Thus, for all $a\in C^1(\Omega),\;w\in C^2(\Omega;\R^n)$ and $\zeta\in C^1_c(\Omega;\R^n),$ it follows that
\begin{equation}\label{0-id}
\int_\Omega  a \cof Dw  :D\zeta\,dx =-\int_\Omega (\cof Dw) Da\cdot \zeta\,dx. 
\end{equation}

\begin{lem}\label{lem-0} The identity (\ref{0-id}) holds for all functions 
\[
w\in W^{1, p}_{loc}(\Omega;\R^n),\quad a\in W^{1,\bar p}_{loc}(\Omega;\R^n), \quad  \zeta\in C_c^1(\Omega;\R^n),
\]
 where $n-1\le p\le \infty$ and $\bar p=\frac{p}{p+1-n}$ are given numbers.
 \end{lem}
 
\begin{proof}\    Note that if $w\in W^{1, p}_{loc}(\Omega;\R^n)$ then $\cof Dw\in L^{\frac{p}{n-1}}_{loc}(\Omega;\M^{n\times n}).$ Since $\bar p$ and $\frac{p}{n-1}$ are H\"older conjugate numbers, it follows,  by the standard approximation arguments,  that the identity (\ref{0-id}) holds for the given functions.  
\end{proof}

   \begin{lem}\label{lem-a} Let $
 H(t)=th(t)-\int_0^t h(s)ds.
$
Then 
\begin{equation}\label{part-0}
 \int_\Omega  h(\det Dw) \cof Dw :D((Dw) \phi)   = \int_\Omega H(\det Dw) \dv\phi  
 \end{equation}
holds for all  $w\in C^2(\Omega;\R^n)$ and 
 $\phi\in C_c^1(\Omega;\R^n).$ 
 \end{lem}
 \begin{proof}\  First we assume $h\in C^1(\R)$ and let $a (x)=h(\det Dw(x)).$ Then $a\in C^1(\Omega)$ and $Da=h'(\det Dw )  D(\det Dw ).$ Take $\zeta=(Dw)\phi\in C_c^1(\Omega;\R^n)$ in (\ref{0-id}) and, from $Dw^T \cof Dw=(\det Dw)I$ and  $H'(t)=th'(t),$ we have \[
  \begin{split}
 \int_\Omega  &h(\det Dw) \cof Dw :D((Dw) \phi)  \\
 &=-\int_\Omega h'(\det Dw ) (\cof Dw )D(\det Dw ) \cdot (Dw)\phi\\
 &= -\int_\Omega  h'(\det Dw) (\det Dw) D(\det Dw ) \cdot \phi   \\
  &= -\int_\Omega  H'(\det Dw) D(\det Dw ) \cdot \phi  \\
   & = -\int_\Omega   D(H(\det Dw) ) \cdot \phi\\&  =  \int_\Omega  H(\det Dw) \dv \phi.  \end{split}
 \]
 This proves  the identity (\ref{part-0}) for $h\in C^1(\R).$ The proof  of (\ref{part-0})  for $h\in C(\R)$ follows by approximating $h$ by $C^1$ functions.  \end{proof}
 
\begin{pro}\label{pro23}  Let $u\in C^2(\Omega;\R^n)$ be a weak solution of equation (\ref{EL-h}). Then $h(\det Du)$ is constant on  $\Omega.$ 
     \end{pro}

 \begin{proof}\  Let $d(x)=\det Du(x);$  then $d\in C^1(\Omega).$ By (\ref{weak-1}) and (\ref{part-0}), it follows that $\int_\Omega H(d(x))\dv\phi \,dx=0$ for all $\phi\in C_c^1(\Omega;\R^n);$  hence  $H(d(x))=C$ is a constant on $\Omega.$  Suppose, on the contrary, that $h(d(x_1))\ne h(d(x_2))$ for  some $x_1,x_2\in \Omega.$ Then $d_1=d(x_1)\ne d_2=d(x_2),$ say $d_1<d_2.$ Since $d(x)$ is continuous, by the intermediate value theorem, it follows that $H(t)=C$  for all $t\in [d_1,d_2].$ Solving for $h(t)$ from the integral equation $th(t)-\int_0^t h(s)ds=C$ on $(d_1,d_2)$, we obtain that $h(t)$ is constant on $(d_1,d_2),$ contradicting $h(d_1)\ne h(d_2).$ 
\end{proof}

\begin{pro}\label{pro24}  Let $p\ge n-1$, $\bar p=\frac{p}{p+1-n},$ and  $u\in W^{1, p}_{loc}(\Omega;\R^n)$ be a weak solution of equation (\ref{EL-h}) such that  $h(\det Du)\in W^{1,\bar p}_{loc}(\Omega).$  Then $h(\det Du)$ is constant   almost everywhere  on   $ \Omega.$ 
     \end{pro}

 \begin{proof}\  With $a=h(\det Du)\in W^{1,\bar p}_{loc}(\Omega)$ and $w=u\in W^{1, p}_{loc}(\Omega;\R^n)$ in Lemma \ref{lem-0} and by (\ref{weak-1}), we have 
 \[
0= \int_\Omega  h(\det Du)\cof Du:D \zeta
  = -\int_\Omega  (\cof Du)D(h(\det Du)) \cdot \zeta 
\]
for all $\zeta\in C_c^1(\Omega;\R^n).$ As a result, we have  $ (\cof Du)D(h(\det Du ))=0$ a.e.\,on $\Omega.$ Thus $D(h(\det Du ))=0$ a.e.\,on the set $\Omega_0=\{x\in\Omega:\det Du(x)\ne 0\}.$ Clearly, $D(h(\det Du ))=0$ a.e.\,on the set $E=\{x\in\Omega:\det Du(x) = 0\}$ because $h(\det Du)=h(0)$ a.e.\,on $E.$ Therefore, $D(h(\det Du ))=0$ a.e.\,on the whole domain $\Omega;$ this proves that $h(\det Du)$ is constant a.e.\,on $\Omega.$  
\end{proof}

\subsection{Change of variables}  The formal calculations leading to Proposition \ref{pro23} do not  work  for the $C^1$ weak solutions of  equation (\ref{EL-h}); however,  the same conclusion still holds, as has been discussed in \cite{KMS} without proof. 
We  give a proof by choosing suitable  test functions based on  change of variables. 

Let us first recall the change of variables  for general Sobolev functions; see, e.g.,  \cite {MM}. 
Let $p>n$ and $w\in W^{1,p}(\Omega;\R^n);$ then $w\in C_{loc}^\alpha(\Omega;\R^n)$ with $\alpha=1-\frac{n}{p}.$  Let $E$ be a measurable subset of $ \Omega.$ Denote by $N(w|E;y)$ the cardinality of the set $\{x\in E: w(x)=y\}.$ Then, for all measurable functions  $g\colon w(\Omega)\to \R,$ the  {\em change of variable formula}:
\begin{equation}\label{cov}
 \int_E g(w(x)) |\det Dw(x)|\, dx =\int_{w(E)} g(y)  N(w|E;y) \,dy
\end{equation}
is valid, whenever one of the two sides is meaningful    (see \cite[Theorem 2]{MM}).

  \begin{thm}\label{thm25}  Let $u\in C^1(\Omega;\R^n)$ be a weak solution of equation (\ref{EL-h}).  Then $h(\det Du)$ is constant on $\Omega.$  
 \end{thm}
 \begin{proof}\  Let $E=\{x\in\Omega:\det Du(x)=0\},$ which is relatively closed in $\Omega.$ There is nothing to prove if $E=\Omega;$ thus we assume $\Omega_0=\Omega\setminus E\ne \emptyset.$ Let $C$ be a component  of the open set $\Omega_0.$  Without loss of generality, we assume $\det Du>0$ on $C.$ 
 Let $x_0\in  C$ and $y_0=u(x_0).$ Since $y_0$ is a regular value of $u$, by the inverse function theorem, there exists an open disk $D=B_\epsilon(y_0)$ such that  \[
 \begin{cases} \mbox{$U=u^{-1}(D)$ is a subdomain of $C;$}\\
 \mbox{$\det Du>0$ on $U;$}\\
 \mbox{$u:U\to D$ is bijective with inverse $v=u^{-1}:D\to U$;}\\
 \mbox{$v\in C^1(D;\R^n).$} \end{cases}
 \] 
 Given any $\phi\in C^1_c(D;\R^n),$ the function $\zeta(x)=\phi(u(x))\in C^1_c(U;\R^n)$ is a test function for (\ref{weak-1})  with $D\zeta(x)=D\phi(u(x))Du(x);$ thus, by the change of variables,   we obtain that
    \[
 \begin{split} 0&=\int_\Omega h(\det Du)\cof Du : D\zeta  \,dx\\
 &=\int_U h(\det Du(x))\cof Du(x) : D\phi(u(x))Du(x) \, dx \\
 &=\int_U h(\det Du(x))\det Du(x) \tr(D\phi(u(x))) \, dx \\
  &  =\int_D h(\det Du(v(y))\dv \phi(y) \,dy. \end{split}
  \]
This holding  for all $\phi\in C^1_c(D;\R^n)$  proves that $h(\det Du(v(y))$ is constant on $D;$ hence $h(\det Du)$ is constant on $U.$  Since $C$ is connected and $h$ is continuous, it follows that $h(\det Du)$ is constant on the relative closure $\bar C$ of $C$ in $\Omega.$ If $E=\emptyset,$ we  have $C=\Omega;$ hence,  $h(\det Du)$ is constant on $\Omega.$ If $E\ne \emptyset,$  we have  $\bar C\cap E\ne \emptyset$ and  thus $h(\det Du)=h(0)$ on $\bar C,$  which proves that $h(\det Du) =h(0)$ on $\Omega_0;$  hence  $h(\det Du) =h(0)$ on the whole $\Omega.$   
  \end{proof}

\begin{cor}\label{cor26}  Let  $u$ be a weak solution of equation (\ref{EL-h}).   Then  $\det Du$ is constant almost everywhere on $\Omega$  if one of the following assumptions holds:
\begin{itemize} 
 \item[(a)]  $u \in C^1(\Omega;\R^n)$ and $h$ is not constant on any intervals. 
  \item[(b)] $u\in W^{1,p}_{loc}(\Omega;\R^n)$, $h(\det Du)\in W^{1,\bar p}_{loc}(\Omega)$, where $n-1\le p\le \infty$ and $\bar p=\frac{p}{p+1-n}$ are given numbers, and $h $ is one-to-one. 
 \end{itemize}
  \end{cor}

 \begin{proof}\  Assuming (a),   by Theorem \ref{thm25}, we have that $h(\det Du)=c$ is constant in $\Omega$ and thus $\det Du(x)\in h^{-1}(c).$   Since $\Omega$ is connected, $\det Du$ is continuous in $\Omega$  and $h^{-1}(c)$ contains no intervals,  it follows that $\det Du$ is constant in $\Omega.$ Assuming (b), by Proposition \ref{pro24}, we have that  $h(\det Du)$ is constant a.e.\,on $\Omega.$ Since $h$ is one-to-one, it follows that  $\det Du$ is constant a.e.\,on $\Omega.$
\end{proof}

  The following result  applies to the weak solutions $u$ of (\ref{EL-h}) such that $u\in C^1(\Omega;\R^n)$ or $u\in W^{1,p}_{loc}(\Omega;\R^n)$ with $h(\det Du)\in W^{1,\bar p}_{loc}(\Omega)$ for some $p\ge n$ and $\bar p=\frac{p}{p+1-n}.$

  \begin{pro}\label{pro26} Assume $h$ is nondecreasing and let $f(t)=\int_0^t h(s)\,ds.$   Suppose that  $u\in W^{1,n}_{loc}(\Omega;\R^n)$ is  a weak solution of (\ref{EL-h}) such that $h(\det Du)$ is constant a.e.\,on $\Omega.$ Then  for all subdomains $G\subset\subset\Omega,$ the inequality
  \[
  \int_G f(\det Du(x))\,dx \le \int_G f(\det Dv(x))\,dx
  \]
  holds for all $v\in W^{1,n}(G;\R^n)$ satisfying $v-u\in W^{1,n}_0(G;\R^n).$
  \end{pro}
  \begin{proof}\  Let $h(\det Du)=\mu$ be a constant a.e.\,on $\Omega.$  Since $f'=h$ is nondecreasing, it follows that $f$ is convex and thus  $f(t)\ge f(t_0)+h(t_0)(t-t_0)$ for all $t,\, t_0\in\R.$  Let $v\in W^{1,n}(G;\R^n)$ satisfy $v-u\in W^{1,n}_0(G;\R^n).$ Then
  \[
  f(\det Dv)\ge f(\det Du) +\mu (\det Dv-\det Du) \quad \mbox{a.e.\,on $\Omega.$}
  \]    Integrating over $G,$ since   $\int_G  \det Dv=\int_G \det Du$, we have
  \[
   \int_G f(\det Dv ) \ge \int_G f(\det Du )  +\mu \int_G (\det Dv-\det Du) =\int_G f(\det Du).
   \]
  \end{proof}

\subsection{Weak solutions in $W^{1,p}(\Omega;\R^n)$ for $p\ge n$} 
It remains open that whether weak solutions $u\in W^{1,n}_{loc}(\Omega;\R^n)$ of equation (\ref{EL-h}) must have a constant $h(\det Du)$ a.e.\,on $\Omega$ without assuming $h(\det Du)\in W^{1,n}_{loc}(\Omega).$ 
In fact, despite the example of very weak solutions in Example \ref{ex1} below, we do not know whether a weak solution  $u\in W^{1,n}_{loc}(\Omega;\R^n)$ of   (\ref{EL-h}) satisfying $h(\det Du)\in L^q_{loc}(\Omega)$ for some $q>1$ will have  a constant $h(\det Du)$ a.e.\,on $\Omega.$  
 However, we have some partial results below.

We first prove a result about the {\em homeomorphisms} in $W^{1,p}(\Omega;\R^n)$ for $p>n;$ such a result may have already been well known.

\begin{lem}\label{lem28} Let  $p>n$ and $u\in W^{1,p}(\Omega;\R^n)$ be such that $u$ is one-to-one on $ \Omega$ and $u(\Omega)$ is a domain in $\R^n.$  Then, either $\det Du\ge 0$ a.e.\,on $\Omega$ or $\det Du\le 0$ a.e.\,on $\Omega.$
\end{lem}
\begin{proof} Since $u$ is one-to-one on $\Omega,$ we have $ N(u|\Omega;y) =1$ for all $y\in u(\Omega).$ 
Let  $\Omega_0=\{x\in\Omega:  \det Du(x)=0\}$, $\Omega_+=\{x\in\Omega:  \det Du(x)>0\}$ and $\Omega_-=\{x\in\Omega:  \det Du(x)<0\}.$ Then 
\[
|\Omega_0|+|\Omega_+|+|\Omega_-|=|\Omega|, \quad |u(\Omega)|=|u(\Omega_0)|+|u(\Omega_+)|+|u(\Omega_-)|.
\]
  On the other hand, by the change of variable formula (\ref{cov}), we have $|u(\Omega_0)|=\int_{\Omega_0} |\det Du|=0.$ Let $\phi\in C^1_c(u(\Omega);\R^n)$ be  given and define $\zeta(x)=\phi(u(x)) \in W^{1,p}_0(\Omega;\R^n).$ Then
\[
 0=\int_{u(\Omega)} \dv\phi(y)\,dy =\int_{u(\Omega_+)} \dv\phi(y)\,dy+ \int_{u(\Omega_-)} \dv\phi(y)\,dy;  
 \]
moreover, since $ \cof Du : D\zeta = (\dv \phi)(u(x)) \det Du =0$ on $\Omega_0,$ we have
\[
\begin{split} 0&=\int_\Omega  \cof Du : D\zeta \,dx  \\
&= \int_{\Omega_0} \cof Du : D\zeta\,dx+  \int_{\Omega_+}  \cof Du : D\zeta\,dx  + \int_{\Omega_-}  \cof Du : D\zeta\,dx  \\
&=   \int_{\Omega_+} (\dv \phi)(u(x)) \det Du\,dx    + \int_{\Omega_-}  (\dv \phi)(u(x)) \det Du\,dx   \\
&=\int_{u(\Omega_+)} \dv \phi(y) \,dy -\int_{u(\Omega_-)} \dv \phi(y) \,dy.
\end{split}
\]
Thus,    it follows that $\int_{u(\Omega_+)} \dv \phi(y) \,dy =\int_{u(\Omega_-)} \dv \phi(y) \,dy =0;$  in particular, 
 \[
 \int_{u(\Omega)}\chi_{u(\Omega_+)}(y)\dv \phi  \,dy=\int_{u(\Omega_+)} \dv \phi \,dy =0\quad\forall\; \phi\in C^1_c(u(\Omega);\R^n).
 \]
  This proves that $\chi_{u(\Omega_+)}$ is constant a.e.\,on $u(\Omega);$  thus  either  $|u(\Omega_+)|=0$ or $|u(\Omega_+)|=|u(\Omega)|.$ If  $|u(\Omega_+)|=0,$ then  $\int_{\Omega_+} |\det Du| \,dx= |u(\Omega_+)| =0;$ since $\det Du>0$ on $\Omega_+,$ this shows that $ |\Omega_+|=0,$ and thus $\det Du\le 0$ a.e.\,on $\Omega.$ If $|u(\Omega_+)|=|u(\Omega)|,$ then $|u(\Omega_-)|=0;$  by a similar proof, this shows  that
  $|\Omega_-|=0,$  and thus $\det Du\ge 0$ a.e.\,on $\Omega.$  This completes the proof. 
  \end{proof}
  
We now have the following result concerning the certain homeomorphism  weak solutions of equation (\ref{EL-h}).  
 
\begin{thm}\label{thm-homeo} Assume $h$ is nondecreasing. Let  $p>n$, $\tilde p=\frac{p}{p-n},$ and let $u\in W^{1,p}(\Omega;\R^n)$ be a weak solution of (\ref{EL-h}) with $h(\det Du)\in L^{\tilde p}(\Omega)$ such that
\begin{equation}\label{inv}
\begin{cases} \mbox{$u$ is one-to-one on $ \Omega,$}\\
\mbox{$ u(\Omega)$ is a domain in $\R^n,$}\\
\mbox{$v =u^{-1}\in C(u(\Omega);\R^n)$ satisfies the {\em Luzin $N$-property:}} \\
\mbox{$|v(B)|=0$ for all $B\subset u(\Omega)$ with $|B|=0.$}
\end{cases}
\end{equation} Then $h(\det Du)$ is a  constant a.e.\,on $\Omega.$
\end{thm}
 \begin{proof}\  Without loss of generality, we assume $h(0)=0;$  thus $h(t)t=|h(t)t|\ge 0$ for all $t\in \R.$ 
From the assumption, we have 
\[
h(\det Du)\cof Du\in  L^{p'}(\Omega;\M^{n\times n});\quad p'= {p}/(p-1).
\]
 Let $\phi\in C^1_c(u(\Omega);\R^n).$ Then $\zeta(x)=\phi(u(x))\in W^{1,p}_0(\Omega;\R^n)$ is  a legitimate test function for (\ref{weak-1}). Since $ N(u|\Omega;y) =1$ for all $y\in u(\Omega),$ by the change of variable formula (\ref{cov}),  we obtain  
    \[
 \begin{split} 0&=\int_\Omega h(\det Du)\cof Du : D\zeta  \,dx\\
 &=\int_\Omega h(\det Du(x))\cof Du(x) : D\phi(u(x))Du(x) \, dx \\
 &=\int_\Omega \tr(D\phi(u(x)))  h(\det Du(x))\det Du(x) \, dx \\ 
 &=\int_\Omega  (\dv \phi)(u(x))  |h(\det Du(v(u(x)))|  |\det Du(x) |\, dx \\
 &  =\int_{u(\Omega)} \dv \phi(y) |h(\det Du(v(y))| N(u|\Omega;y)\,dy\\
 &  =\int_{u(\Omega)} \dv \phi(y) |h(\det Du(v(y))| \,dy. \end{split}
  \]
This holding for all $\phi\in C^1_c(u(\Omega);\R^n)$  proves that  $|h(\det Du(v(y))|$ is constant a.e.\,on $u(\Omega).$ Since $|v(B)|=0$ for all null sets  $B\subset u(\Omega)$, we have that $|h(\det Du)|=\lambda$ is a constant a.e.\,on $\Omega.$ If $\lambda=0$ then $h(\det Du)=0$ a.e.\,on $\Omega.$ Now assume $\lambda>0$ and let  
\[
E_+=\{x\in\Omega: h(\det Du(x))=\lambda\},\quad E_-=\{x\in\Omega: h(\det Du(x))=-\lambda\}.
\]
Then $|E_+|+|E_-|=|\Omega|.$  Since $\lambda>0$ and $h$ is nondecreasing,  we  have  $\det Du> 0$ a.e.\,on $E_+$ and $\det Du<0$ a.e.\,on $E_-.$ Finally, by Lemma \ref{lem28}, we conclude  that either $|E_+|=0$ or $|E_-|=0,$ which  proves the theorem.  
\end{proof}

\begin{remk}  A sufficient condition  for the invertibility of Sobolev functions has been given in \cite{Ba2}. For example, suppose that $\Omega$ is a bounded Lipschitz domain and  $u_0\in C(\bar\Omega;\R^n)$ is such that $u_0$ is one-to-one on $\bar\Omega$ and $u_0(\Omega)$ satisfies the cone condition. Then,  condition (\ref{inv}) is satisfied for a function $u\in W^{1,p}(\Omega;\R^n)$ with $p>n$ provided that
\[
\begin{cases} u|_{\partial\Omega}=u_0,\\
\mbox{$\det Du >0$  a.e.\,on $\Omega,$}\\
\int_\Omega |\cof Du|^q(\det Du)^{1-q} \,dx<\infty \;\; \mbox{for some $q>n$.} 
\end{cases}
\]
See  \cite{Kr} for recent studies and more references  in this direction.
\end{remk}

The following result concerns the weak solutions with certain linear Dirichlet boundary conditions.

\begin{pro}  Let $h(0)=0$,  $p\ge n$, $\tilde p=\frac{p}{p-n},$ and $u\in W^{1,p}(\Omega;\R^n)$ with $h(\det Du)\in L^{\tilde p}(\Omega)$ be a weak solution of (\ref{EL-h}) satisfying  the Dirichlet boundary condition $u|_{\partial\Omega}=Ax,$ where $A\in \M^{n\times n}$ is given. Let 
\[
\lambda=\int_\Omega h(\det Du)\det Du\,dx, \quad B=\int_\Omega h(\det Du) \cof Du\,dx.
\]
 Then  $BA^T=\lambda I.$  Moreover, if $\det A=0$ and $h$ is one-to-one, then $\det Du=0$ a.e.\,on $\Omega,$ and thus $B=0.$  
\end{pro}
 \begin{proof}\  For $P\in \M^{n\times n},$ the function  $\zeta(x)=P(u(x)-Ax)\in W^{1,p}_0(\Omega;\R^n)$ is a legitimate test function for (\ref{weak-1}), which, from $D\zeta =PDu-PA$ and $\cof Du: PDu =( \tr P )\det Du,$ yields   that
\[
(\tr P) \int_\Omega h(\det Du)\det Du \,dx= \int_\Omega h(\det Du)\cof Du:PA\,dx.
\]
This is simply  $(\lambda I-BA^T):P=0.$ Since $P\in \M^{n\times n}$ is arbitrary, we have $BA^T=\lambda I.$  If $\det A=0,$ then $\lambda=0.$ Furthermore, if $h$ is one-to-one, then $\lambda=0$ implies $\det Du=0$ a.e.\,on $\Omega$ and thus $B=0.$ 
 \end{proof}

\begin{remk} (i) Assume $h(0)=0$ and $h$ is one-to-one. If  $\det A\ne 0$, then  $B=\mu h(\det A)\cof A,$ where $\mu=\frac{\lambda}{h(\det A)\det A}>0.$ It remains open whether $\mu=1.$ Note that if $\mu\ne 1$ then $\det Du$ cannot be a constant a.e.\,on $\Omega.$

(ii) There are many (very) weak solutions $u\in W^{1,p}(\Omega;\R^n)$ of equation (\ref{EL-h}) satisfying  $u|_{\partial\Omega}=Ax$  for some $p<n$ such that $\det Du$ is not constant a.e.\,on $\Omega.$ See Example \ref{ex1} below.
\end{remk} 

  \section{Radially symmetric solutions} \label{sec-radial}
  
  Let $B=B_1(0)$ be the open unit ball in $\R^n.$ We consider the radially symmetric or {\em radial} functions 
  \begin{equation}\label{radial}
 u(x)= \phi(|x|) x,
    \end{equation}
  where $\phi\colon (0,1)\to\R$ is weakly differentiable. With $r=|x|$ and $ \omega=\frac{x}{|x|}$,
 we have
  \begin{equation}\label{det}
  \begin{cases}     Du(x) =\phi(r) I + r\phi'(r)\, \omega\otimes \omega, \\
    \det Du(x) = \phi(r)^n +r\phi'(r) \phi(r)^{n-1}, \\
     \cof Du(x) = \alpha(r)I+\beta(r) \, \omega\otimes \omega, \end{cases}
    \end{equation}
 for a.e.\,$x\in B,$ where $\alpha(r)=\phi^{n-1}+r\phi^{n-2}\phi'$ and $\beta(r)=-r\phi^{n-2}\phi'.$ 
  
  \subsection{Some properties of radial functions} We study some properties of radial functions pertaining to the  equation (\ref{EL-h}).  
    
    \begin{lem}\label{S-mean}
   Let $p\ge 1$ and $v\in L^p_{loc}(B\setminus \{0\}).$ Define $\tilde v=M(v)\colon (0,1)\to \R$ by setting
    \begin{equation}\label{psi}
\tilde v(r)  =M(v)(r)= \aint_{S_r}v(x) \,d\sigma_r=\frac{1}{\omega_n} \int_{S_1}v(r\omega) \,d\sigma_1,
     \end{equation}
 where $S_r=\partial B_r(0)$, $d\sigma_r=d\mathcal H^{n-1}$ denotes the $(n-1)$-Hausdorff measure on $S_r,$ and $\omega_n=\mathcal H^{n-1}(S_1).$   Then $\tilde v\in L^p_{loc}(0,1).$ Furthermore, if $v\in W^{1,p}_{loc}(B\setminus \{0\}),$ then $\tilde v\in W^{1,p}_{loc}(0,1)$ with
 \[
 \tilde v'(r)= \frac{1}{\omega_n} \int_{S_1}Dv(r\omega)\cdot \omega \,d\sigma_1 =r^{-1}M(Dv\cdot x)(r) \quad a.e.\; r\in (0,1).
 \]
Since $W^{1,p}_{loc}(0,1)\subset  C(0,1)$, $\tilde v$ can be identified as  a continuous function in  $(0,1)$ if $v\in W^{1,p}_{loc}(B\setminus \{0\}).$
 \end{lem}
    
    \begin{proof}\  Note that 
    \[
     |\tilde v(r)| \le \aint_{S_r}|v| \,d\sigma_r \le \left ( \aint_{S_r}|v|^p \,d\sigma_r\right)^{1/p}.
     \]
    Thus, for all $0<a<b<1$,
    \[
    \int_a^b \omega_nr^{n-1}|\tilde v|^p\,dr\le \int_a^b \left ( \int_{S_r}|v|^p \,d\sigma_r\right)dr=\int_{a<|x|<b} |v(x)|^p\,dx<\infty.
    \]
    This proves $\tilde v\in L^p_{loc}(0,1).$  Now assume $v\in W^{1,p}_{loc}(B\setminus \{0\})$ and let
    \[
    g(r)=  \frac{1}{\omega_n} \int_{S_1}Dv(r\omega)\cdot \omega \,d\sigma_1 =r^{-1}M(Dv\cdot x)(r) \in L^p_{loc}(0,1).
    \]
    Let $0<a<b<1$ and $\eta\in C^\infty_c(a,b).$ Then
    \[
    \int_a^b \tilde v(r)\eta'(r)\,dr=\frac{1}{\omega_n} \int_a^b r^{1-n}\eta'(r) \Big( \int_{S_r} v \,d\sigma_r\Big) dr \]
    \[
   =\frac{1}{\omega_n} \int_{a<|x|<b} |x|^{1-n}\eta'(|x|)v(x)\,dx=\frac{1}{\omega_n} \int_{a<|x|<b} D(\eta(|x|))\cdot (v|x|^{-n}x)\,dx
   \]
   \[
  =-\frac{1}{\omega_n} \int_{a<|x|<b} \eta(|x|) \dv(v|x|^{-n}x)\,dx=-\frac{1}{\omega_n} \int_{a<|x|<b} \eta(|x|) Dv\cdot (|x|^{-n}x)\,dx\]
  \[
=  -\frac{1}{\omega_n} \int_a^b \eta(r) r^{-n} \Big (\int_{S_r} Dv(x)\cdot xd\sigma_r\Big )\,dr = -\int_a^b g(r)\eta(r)\,dr.
    \]
 This proves $g=\tilde v'.$   
     \end{proof}
 
  \begin{pro}\label{pro29} Assume $\phi\in W_{loc}^{1,1}(0,1)$ and $u(x)=\phi(|x|)x.$ Let $p\ge 1$ and $p'=\frac{p}{p-1}.$ Suppose that  $q(r)$ is a measurable function on $(0,1)$ such that  
  \begin{equation}\label{integ}
  q(|x|)\cof Du\in L^p_{loc}(B\setminus \{0\}).
  \end{equation}
Then, for all $0<a<b<1$ and $v\in W_{loc}^{1,p'}(B;\R^n),$ it follows that 
 \begin{equation}\label{id-1}
 \int_{a<|x|<b} q(|x|)\cof Du:Dv  \,dx = \omega_n \int_a^b q(r)(r^{n-2}\phi^{n-1}\psi)' \,dr,  
  \end{equation}
  where $\psi(r)=M(v\cdot x)(r)=\frac{r}{\omega_n} \int_{S_1}v(r\omega)\cdot \omega \,d\sigma_1.$
 \end{pro}
 
 \begin{proof}\    Note that $q(|x|)\cof Du:Dv\in L^1_{loc}(B\setminus \{0\}).$ Since $ \cof Du: Dv =\alpha(r)\dv v + \beta(r) Dv:(\omega\otimes \omega),$   we have
 \begin{equation}\label{id-2}
\begin{split}
\int_{a<|x|<b} q(|x|)\cof Du:Dv  \,dx  &=\int_a^b \Big( \int_{S_r} q(|x|)\cof Du:Dv \,d\sigma_r\Big )dr \\
= \int_a^b q(r) \Big  (\alpha(r) \int_{S_r} \dv v\,d \sigma_r  &+ \beta(r) \int_{S_r}Dv:(\omega\otimes \omega) \,d\sigma_r \Big ) dr.
\end{split}
\end{equation}
We now compute the two spherical integrals. First, for all $a<t<1,$ by the divergence theorem,
 \[
 \int_a^t  \int_{S_r} \dv v\,d\sigma_r \,dr =\int_{a<|x|<t}\dv v\,dx\]
 \[
 =\int_{S_t} v\cdot \frac{x}{t}\,d\sigma_t-\int_{S_a} v\cdot \frac{x}{a}\,d\sigma_a =\omega_n t^{n-2}\psi(t)-\omega_n a^{n-2}\psi(a).
 \]
Hence
 \[
  \int_{S_r} \dv v\,d\sigma_r = \omega_n (r^{n-2}\psi(r))' \quad a.e. \; r\in (0,1). 
  \]
Second, since $D(v\cdot x) \cdot x =Dv : (x\otimes x) +v\cdot x,$ we have
\[
\psi'(r)=   r^{-1} M(D(v\cdot x) \cdot x)= r^{-1} M(Dv : (x\otimes x)) + r^{-1} M(v\cdot x).\]
Thus $M(Dv : (x\otimes x)) =r \psi' - \psi$ and hence
\[
 \int_{S_r}Dv:(\omega\otimes \omega) \,d\sigma_r =\omega_n r^{n-3}M(Dv : (x\otimes x))  =\omega_n r^{n-2} \psi' -\omega_n r^{n-3}\psi 
 \]
 \[
  = \omega_n (r^{n-2}\psi)' -(n-1)\omega_n r^{n-3}\psi.
 \]
Since $\alpha+\beta=\phi^{n-1}$ and $\beta=-r\phi^{n-2}\phi'$, elementary computations lead to 
\[
\alpha(r) \int_{S_r} \dv v\,d\sigma_r+ \beta(r) \int_{S_r}Dv:(\omega\otimes \omega) \,d\sigma_r
\]
\[
= \omega_n \alpha (r^{n-2}\psi)' +  \omega_n \beta  [(r^{n-2}\psi)' -(n-1) r^{n-3}\psi]=\omega_n (\phi^{n-1}r^{n-2}\psi)'
\]
for a.e.\,$r\in (0,1).$ Finally, (\ref{id-1}) follows from (\ref{id-2}). 
 \end{proof}

 \begin{thm} Assume $h\in C(\R)$   is one-to-one, $p\ge \frac{n}{n-1}$ and $\phi \in W^{1,p}_{loc}(0,1).$  Let $u(x)=\phi(|x|)x$ be a weak solution of (\ref{EL-h}) such that
 \begin{equation}\label{dual-p}
 h(\det Du)\cof Du\in L_{loc}^\frac{p}{p-1}(B;\M^{n\times n}).
 \end{equation}   
 Then either $\phi\equiv 0,$ or 
 \[
 \phi(r)=\left(\lambda + \frac{c}{r^n}\right)^{1/n}\ne 0\quad\forall\; 0<r<1,
 \]
 where $\lambda$ and $c$ are constants.  (When $n$ is even, we need $\lambda + \frac{c}{r^n}> 0$ in $(0,1)$ and  there are two nonzero branches  of the $n$-th roots.)
   \end{thm}
  
 \begin{proof}\   Let $S=\{r\in (0,1):\phi(r)\ne 0\};$ then $S$ is  open.   If $S=\emptyset,$ then $\phi\equiv 0.$ Assume $S$ is  nonempty.   Let $(a,b)$ be   a component of $S.$  Let $\eta(r)\in C_c^\infty(0,1)$ be any function with  compact support contained in $(a,b).$ Define the radial function
\[
 \zeta(x)=\begin{cases} \dfrac{\eta(r)  }{r^{n}\phi(r)^{n-1}}  \,x & \mbox{if $r=|x|\in (a,b),$}\\
 0 & \mbox{otherwise.}\end{cases}
 \]
Then $\zeta\in W^{1,p}(B;\R^n)$ with $\supp \zeta\subset \{a<|x|<b\}.$ Let $\psi=M(\zeta\cdot x);$ then $
r^{n-2}\phi^{n-1}\psi=\eta.
$
 Let  $ \det Du=\phi^n +r\phi' \phi^{n-1}=:d(r).$ By assumption (\ref{dual-p}),  $\zeta$ is a legitimate test function for equation (\ref{EL-h}); thus,  by (\ref{id-1}), we obtain that 
\[
0= \int_B h(\det Du)\cof Du: D\zeta\,dx \]
 \[
 =  \int_a^b h(d(r)) (r^{n-2}\phi^{n-1}\psi)'\, dr  =\int_a^b h(d(r))\eta'(r)dr.
\]
This holds  for all  $\eta\in C^\infty_c(a,b);$ thus $h(d(r))$ is constant  a.e.\,in $(a,b)$. As $h$ is one-to-one, we have that  $d(r)$ is constant a.e.\,in $(a,b).$    Assume $d(r) =\phi^n +r\phi' \phi^{n-1}= \lambda$ in $(a,b).$ Solving the differential equation  we obtain that
    \[
 \phi(r)\ne 0,\quad   \phi(r)=\left(\lambda + \frac{c}{r^n}\right)^{1/n} \quad (a<r<b).
   \]
If one of $a$ and $b$ is inside $(0,1),$ then $\phi=0$ at this point; but in this case, $\phi\notin W^{1,q}_{loc}(0,1)$ for any $q\ge \frac{n}{n-1}.$ So $(a,b)=(0,1);$ this completes the proof.
 \end{proof}

\subsection{Very weak solutions}   
We consider some examples of (very) weak solutions of (\ref{EL-h}) in $W^{1,p}(B;\R^n)$ with $p<\frac{n}{n-1}.$

\begin{example}\label{ex1} Let $n\ge 2$, $0<a\le b<1$ and $\lambda_1\ne \lambda_2.$  Let  $u=\phi(|x|)x$, where  
\begin{equation}\label{ex-1}
\phi (r)=\begin{cases}  [\lambda_1 (r^n-a^n)]^{1/n}/{r}  & (0< r \le a),\\ 
0 & (a< r \le b),\\ 
 [\lambda_2 (r^n-b^n)]^{1/n}/{r} & (b < r\le  1).
 \end{cases}
  \end{equation}
Then $u \in W^{1,p}(B;\R^n)$  for  all $1\le p< \frac{n}{n-1}$ and $u$  is a weak solution of equation (\ref{EL-h}) in $B$ satisfying the Dirichlet boundary condition 
\[
u|_{\partial B}=Ax=[\lambda_2 (1-b^n)]^{1/n}x,
\]
 but  $\det Du$ is not a constant on $B.$  
Moreover, if $\lambda_2=0$, then $u\chi_B$  is a weak solution of (\ref{EL-h}) on the whole $\R^n$; namely, 
\begin{equation}\label{r-2}
\int_B    h(\det Du)   \cof Du:D\zeta =0\quad\forall\,\zeta\in C^1(\R^n;\R^n).
\end{equation}
\end{example}

\begin{proof}\  If $\lambda_1=0$, then only $Du$ blows up at $r=|x|=b,$ with 
$
|Du(x)|\approx |\phi'(r)|\approx  |r-b|^{\frac{1}{n}-1}$ and $|\cof Du(x)|\approx |\phi(r)|^{n-2}|\phi'(r)|\approx |r-b|^{-\frac{1}{n}}$ near $r=b.$ Hence, in this case,  $u \in W^{1,p}(B;\R^n)$  for  $1\le p< \frac{n}{n-1}$  and $\cof Du\in L^q(B;\M^{n\times n})$ for all $1\le q<n.$ 

If $\lambda_1\ne 0$, then  $u$ and $Du$  also blow up at $x=0$ and $Du$ blows up at $r=a$ and $r=b.$ In this case, the similar blow-up estimates show  that $u \in W^{1,p}(B;\R^n)$  and $\cof Du\in L^p(B;\M^{n\times n})$  for  $1\le p< \frac{n}{n-1}.$  

Hence, in all cases, $u \in W^{1,p}(B;\R^n)$  and $h(\det Du) \cof Du \in L^p(B;\M^{n\times n})$ for  all $1\le p< \frac{n}{n-1}.$  
Given any $\zeta\in C^1(\R^n;\R^n),$ let $ \psi(t)= M(\zeta\cdot x)(t)=\frac{t}{\omega_n} \int_{S_1}\zeta (t\omega)\cdot \omega \,d\sigma_1.$ Then 
\[
 \lim_{t\to 0^+} (t^{-1}\psi(t))= \lim_{t\to 0^+}   \frac{1}{\omega_n} \int_{S_1}\zeta (t\omega)\cdot \omega \,d\sigma_1=0.
\]
By Proposition \ref{pro29}, we have 
\[
 \begin{split}    \int_B  & h(\det Du)   \cof Du:D\zeta  
 =\lim_{t\to 0^+ } \int_{t<|x|<1} h(\det Du) \cof Du:D\zeta \\
 &= h(\lambda_1) \lim_{t\to 0^+} \int_{t<|x|<a}   \cof Du:D\zeta 
   + h(\lambda_2)  \int_{b<|x|<1}  \cof Du:D\zeta   \\
 & =-\omega_n  h(\lambda_1) \lim_{t\to 0^+} t^{n-2}\phi(t)^{n-1}\psi(t) +\omega_n  h(\lambda_2)  (t^{n-2}\phi(t)^{n-1}\psi(t))|_b^1\\
 &=\omega_n  h(\lambda_2)  \phi(1)^{n-1}\psi(1).\end{split}
  \]
If $\zeta\in C^1_c(B;\R^n)$, then $\psi(1)=0$; this proves that $u$ is a weak solution of   (\ref{EL-h}). Moreover, if $\lambda_2=0$ then $\phi(1)=0;$ in this case, we obtain (\ref{r-2}).  
\end{proof}

\begin{example}\label{ex2} Let $\Omega$ be any bounded domain in $\R^n$,  $\{\bar B_{r_i}(c_i)\}_{i=1}^\infty$ be a family of disjoint closed balls in $\Omega,$ and
 $ u_i(x)= \phi_i(|x|)x$ be the radial function on $B=B_1(0),$ where $\phi_i(r)$  is defined by (\ref{ex-1}) with $\lambda_1=t_i\ne 0$, $a=a_i\in (0,1)$  and $\lambda_2=0.$  By choosing suitable $t_i$ and $a_i$, we assume
\[
\|u_i\|_{W^{1,p}(B)}+\|\cof Du_i\|_{L^p(B)} \le M\quad \forall\, i=1,2,\cdots 
\]
for some constants $M>0$ and $1\le p< \frac{n}{n-1}.$  Define
\[
u(x)=\begin{cases} r_i u_i(\frac{x-c_i}{r_i})  & x\in B_i=B_{r_i}(c_i),\\
0 & x\in \Omega\setminus \cup_{i=1}^\infty B_i.
\end{cases}
\]
Then $u \in W^{1,p}(\Omega;\R^n)$, $\cof Du\in L^p(\Omega;\M^{n\times n})$ and  
$
|\{x\in\Omega: \det Du(x)=t_i\}|\ge a_i^n |B_i|$ (with equality holding  if $t_i\ne t_j$ for all $i\ne j$). 

Moreover, $u$ is a weak solution of (\ref{EL-h}) on $\Omega;$ in fact, $u\chi_\Omega$ is  a weak solution of (\ref{EL-h}) on $\R^n.$ To see this, given any $\zeta\in C^1(\R^n;\R^n),$ we observe that
\[
 \begin{split}    \int_\Omega  & h(\det Du)   \cof Du:D\zeta  \,dx
 =\sum_{i=1}^\infty \int_{B_i} h(\det Du) \cof Du:D\zeta\,dx \\
 &=\sum_{i=1}^\infty r_i^n \int_{B} h(\det Du_i(z)) \cof Du_i(z):(D\zeta)(c_i+r_iz)\,dz  \\
&=\sum_{i=1}^\infty \int_{B} h(\det Du_i(z)) \cof Du_i(z):D\zeta_i(z)\,dz =0,
\end{split}
\]
where $\zeta_i(z) =r_i^{n-1}  \zeta(c_i+r_iz)\in C^1(\R^n;\R^n).$ By (\ref{r-2}), we have
\[
 \int_{B} h(\det Du_i(z)) \cof Du_i(z):D\zeta_i(z)\,dz=0\quad \forall\,i=1,2,\cdots.
 \]
 Hence $ \int_\Omega   h(\det Du)   \cof Du:D\zeta  \,dx=0$  for all $\zeta\in C^1(\R^n;\R^n).$
    \end{example}
 
\subsection{Non-quasimonotonicity}  Quasimonotonicity is an important condition related to the existence and regularity of weak solutions of certain systems of partial differential  equations; see  \cite{CZ92, Fu87,  Ha95, KM06,   La96,  Zh86}. 

\begin{defn} A function $\sigma \colon\M^{n\times n}\to \M^{n\times n}$ is said to be  {\em quasimonotone  at $A\in \M^{n\times n}$}  provided that
\[
\int_\Omega \sigma(A+ D\phi(x)):D\phi(x)\,dx \ge 0\quad \forall\;  \phi\in C_c^1(\Omega;\R^n).
\]
(This condition is independent of the domain $\Omega.$) 
\end{defn}
We have the following result, which  holds for the model case $h(t)=t;$ the result also holds for a more general class  of functions  including $h(t)=e^t$, but we do not intend to  dwell on the generality.

\begin{thm} Let $n\ge 2$ and $h\in C^1(\R)$ be  such that
\begin{equation}\label{growth1}
 \lambda  t^{k_1}  \le h'(t)\le \Lambda(t^{k_2}+1)\quad \forall \, t\ge 0,
 \end{equation}
  where $\Lambda>\lambda>0$,  $k_1$ and $k_2$ are constants  such  that    
 $
  0\le k_1\le k_2<k_1+1.$  Then $\sigma(A)=h(\det A)\cof A$ is not quasimonotone at $I\in \M^{n\times n}.$ 
\end{thm}

 \begin{proof}\  Let $\Omega$  be the  unit ball in $\R^n.$ We show that there exists a {\em radial} function 
$
 \phi(x)=\rho(|x|)x,
$
  where $ \rho\in W^{1,\infty}(0,1)$ with $ \rho(1)=0,$ such that 
  \begin{equation}\label{qm}
  \int_\Omega \sigma(I+ D\phi(x)):D\phi(x)\,dx<0.
  \end{equation}
Using the spherical coordinates, we compute that
\[
 \int_\Omega \sigma(I+ D\phi(x)):D\phi(x)\,dx =\omega_n \int_0^1 P(r)\,dr,
 \]
   where 
 \[
 P(r)=h\Big ((1+ \rho)^{n} + (1+ \rho)^{n-1} \rho' r\Big )  \Big ( n\rho (1+\rho)^{n-1} +(1+n\rho)(1+\rho)^{n-2} \rho' r\Big)r^{n-1}
 \]
 \[
 =h(A+B)(C+D),
 \]
 with $ A =(1+\rho)^n,\; B =(1+\rho)^{n-1}\rho' r,\;  C =n\rho (1+\rho)^{n-1}r^{n-1},$ and $
 D =(1+n\rho)(1+\rho)^{n-2}\rho'r^n.$ We write  $h(A+B)=h(A+B)-h(A)+h(A)=EB+h(A),$ where
\[
E=  \int_0^1h'(A +tB )dt.
\]
Thus $P=EBC+EBD+h(A)C+h(A)D.$ 

  Let $0<a<1$ be fixed and $b=a-\epsilon$ with $0<\epsilon<a$ sufficiently small. Define
   \begin{equation}\label{rho}
 \rho=\rho_\epsilon(r)=\begin{cases} -1, & 0\le r\le b,\\
 \frac{n-1}{n\epsilon }(r-a)-\frac{1}{n}, & b\le r \le a,\\
 \frac{1}{n(1-a)} (r-1), &a\le r\le 1.
 \end{cases}
 \end{equation}
(See  Figure \ref{fig1b}.)   Then, with $P(r)=P_\epsilon(r)$, we  have
\begin{equation}\label{m1} 
\int_0^1 P_\epsilon (r)\,dr=\int_b^a P(r)\,dr+\int_a^1 P(r)\,dr,\quad 
\Big |\int_a^1 P(r)\,dr \Big |\le M_1,
\end{equation}
where $M_1$ (likewise, each of the $M_k$'s  below) is a positive constant independent of $\epsilon.$  
     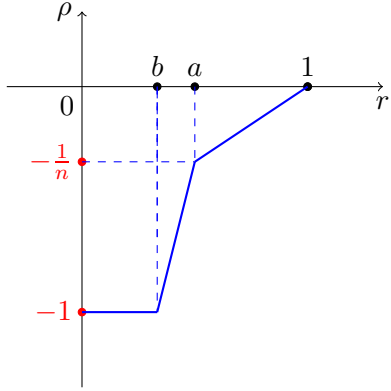
\begin{figure}[ht]     
 \begin{tikzpicture}  	
   \draw[->]  (-1,0)--(4,0) node[below]{$r$};
			\draw[->](0,-4)--(0,1) node[left]{$\rho$};
			  
			     \draw[fill] (1.5,0)  circle (0.05); 
			  			  			     			      \draw[fill] (3,0)  circle (0.05); 
			    
			        \draw[fill] (1,0)  circle (0.05); 
			         \draw[] (3,0) node[above]{$1$};    \draw[] (1.5,0) node[above]{$a$};    \draw[] (1,0) node[above]{$b$};
			         \draw[fill] (3,0)  circle (0.05);

			             \draw[red] (0,-1) node[left]{$-\frac{1}{n}$};   \draw[red] (0,-3) node[left]{$-1$};
			                       \draw[fill, red] (0,-1)  circle (0.05);  \draw[fill, red] (0,-3)  circle (0.05); 
		 
			\draw[blue,thick](1.5,-1)--(3,0); 	\draw[blue,thick](1.5,-1)--(1,-3);\draw[blue,thick](0,-3)--(1,-3);
			\draw[blue,dashed](1,-2)--(1,0); 	\draw[blue,dashed](1,-1)--(1,0); \draw[blue,dashed](1,0)--(1,-3); \draw[blue,dashed](0,-1)--(1.5,-1); \draw[blue,dashed](1.5,0)--(1.5,-1); 
			  		  \draw[] (-0.2,0) node[below] {0};
 		                \end{tikzpicture}
 \caption{The graph of $\rho=\rho_\epsilon(r),$ where $0<a<1$ is fixed and $b=a-\epsilon$ with $\epsilon \in (0,a)$ sufficiently small.}
   \label{fig1b}
 \end{figure}
For all $b<r<a$ we have $\rho'=\frac{n-1}{n\epsilon}$,  $n\rho+1=\frac{n-1}{ \epsilon } (r-a)$ and $ \rho+1 =\frac{ n-1 }{n \epsilon} (r-b).$  Hence $0<A<1$ and $B> 0$ on $(b,a)$; moreover,  
 \begin{equation}\label{m2}
 \Big |\int_b^a \Big (  h(A)C+h(A)D \Big )\,dr \Big |\le M_2.
 \end{equation} 
 Moreover,
 \[
 \lambda\int_0^1 (A+tB)^{k_1} dt \le E\le \Lambda \left (1+\int_0^1(A+tB)^{k_2} dt\right ).
 \]
Since
\[
 \int_0^1  (A+t B)^k dt = \epsilon^{-k} \int_0^1  \big (\epsilon A+ (1+\rho)^{n-1}\frac{n-1}{n} r t \big )^k dt, 
 \]
 it follows that for all $k\ge 0$
\[
M_3 \epsilon^{-k}  (1+\rho)^{k(n-1)}r^k \le  \int_0^1 (A+t B)^k dt  \le M_4 \epsilon^{-k}.
 \]
Thus
 \begin{equation}\label{m3}
 \Big |\int_b^a  EBC  \,dr \Big |\le M_5 (1+ \epsilon^{-k_2}).
 \end{equation}
 Since $BD\le 0$ on $(b,a)$, we have
 \[
 \int_b^a EBD\,dr \le \lambda \int_b^a M_3 \epsilon^{-k_1}  (1+\rho)^{k_1(n-1)}r^{k_1} BD\,dr \]
 \[
 =\lambda M_3 \epsilon^{-k_1}   \int_b^a(1+\rho)^{k_1(n-1)+2n-3}(n\rho +1)\rho'^2 r^{k_1+n+1}\,dr  \]
\[
 \le \frac{ b^{k_1+n+1}M_6}{\epsilon^{2n+k_1n}}  \int_b^a (r-b)^{k_1(n-1)+2n-3}(r-a) \,dr  =-M_7\epsilon^{-k_1-1}.
  \]
This, combined with  (\ref{m1})-(\ref{m3}), proves that 
\[
\lim_{\epsilon\to 0^+} \left (\epsilon^{k_2} \int_0^1 P_\epsilon (r)\,dr\right ) \le \lim_{\epsilon\to 0^+} (M_5 -M_7 \epsilon^{k_2-k_1-1}) =-\infty.
\]
Consequently, $\int_0^1 P_\epsilon (r)\,dr<0$ if  $\epsilon\in (0,a)$ is sufficiently small; this establishes (\ref{qm}).
\end{proof}

\section{The two-dimensional case} \label{sec-2-D}

The rest of the paper is devoted to  the Euler-Lagrange equation (\ref{EL-h})  in the two-dimensional case.  It is now standard \cite{KMS,MSv2} that  equation (\ref{EL-h})   can be  studied through  the  first-order partial differential relation:
\begin{equation}\label{PDR}
DU(x)\in  K = \left \{\begin{bmatrix} A\\ h(\det A) J A\end{bmatrix} : A\in \M^{2\times 2}\right\} \quad \mbox{a.e. in $\Omega,$}
\end{equation}
 for functions $U=(u,v)\colon \Omega \to \R^4,$ where $J=\begin{bmatrix}0&-1\\1&0\end{bmatrix};$ so  $\cof A=-JAJ$ for all $A\in \M^{2\times 2}.$ In particular, if $U=(u,v)\in W^{1,1}_{loc}(\Omega;\R^4)$ is a  solution  of  (\ref{PDR}),  then $h(\det Du)\cof Du=-Dv J$ a.e.\,on $\Omega$; thus,  $u\in W^{1,1}_{loc}(\Omega;\R^2)$ is a  weak solution of (\ref{EL-h}).

If  we define $W=(u,-Jv)$ for $U=(u,v)\colon \Omega \to \R^4,$  then the relation $DU(x)\in K $ is equivalent to the relation  \begin{equation}\label{set-K}
D W(x)\in  \K  = \left \{\begin{bmatrix} A\\ h(\det A)  A\end{bmatrix} : A\in \M^{2\times 2}\right\} \quad \mbox{a.e. in $\Omega.$}
  \end{equation} 
 
 \subsection{Some algebraic structures of the set $\K$} 
 
 Various {\em semi-convex hulls} of $\K$ have been studied in \cite{KMS}; we refer to \cite{D,KMS} for  definitions and further properties of these semi-convex hulls. For example,  in   the model case when $h(t)=t$, it has been proved in \cite{KMS} that the {\em rank-one convex hull} $\K^{rc}=\K$; therefore, the set $\K$ does not support any open structures of $T_N$-configurations \cite{KMS, Ta93}, which makes the construction  of counterexamples  in \cite{MSv2,Sz04,Ya1} impossible using such a set $\K.$     The following result give some algebraic restrictions on  the {\em quasiconvex hull} $\K^{qc}$ of $\K.$    
 
\begin{pro} Assume $h(0)=0$ and $h$ is one-to-one.    Let $A,B\in \M^{2\times 2}.$ Suppose that there exist  sequences $\{u_n\}$ and $\{v_n\}$ uniformly bounded in $W^{1,\infty}(\Omega;\R^2)$   such that
\begin{equation}\label{approx}
\begin{split}  u_n|_{\partial\Omega}=Ax, \quad &  v_n|_{\partial\Omega}=Bx,\\
  \lim_{n\to \infty}\int_\Omega |h(\det Du_n)   Du_n  & - Dv_n | \,dx=0,\end{split}
\end{equation} 
 Then $B=\mu h(\det A)  A$ for some $\mu >0.$ 
\end{pro}

\begin{proof}\  If $X\in \M^{2\times 2}$ has $\chi^1$ and $\chi^2$ as its first and second rows, then we use $\chi^1\wedge \chi^2$ to denote $\det X.$  Let $\alpha^i, \beta^i$ be  the $i$th row of  $A, B,$  respectively,  and  let $\alpha_n^i$ and $ \beta_n^i$ be  the $i$th row of  $Du_n$ and $Dv_n,$  respectively, for $i=1,2.$
 From  the uniform boundedness of $\{|Du_n|\}$ and $\{|Dv_n|\}$ and the limit in (\ref{approx}), we have
 \begin{equation}\label{approx-1} 
\begin{split}  \lim_{n\to \infty}\int_\Omega |h(\alpha_n^1 \wedge \alpha_n^2)(\alpha_n^1\wedge \alpha_n^2) -\alpha_n^1\wedge \beta_n^2| \,dx=0, \\
 \lim_{n\to \infty}\int_\Omega |h(\alpha_n^1 \wedge \alpha_n^2)(\alpha_n^1\wedge \alpha_n^2) -\beta_n^1\wedge \alpha_n^2| \,dx=0,\\
 \lim_{n\to \infty}\int_\Omega |h(\alpha_n^1 \wedge \alpha_n^2)(\alpha_n^2\wedge \alpha_n^2) -\alpha_n^2\wedge \beta_n^2| \,dx=0, \\
 \lim_{n\to \infty}\int_\Omega |h(\alpha_n^1 \wedge \alpha_n^2)(\alpha_n^1\wedge \alpha_n^1) -\alpha_n^1\wedge \beta_n^1| \,dx=0.
 \end{split}
 \end{equation}
Since $\alpha_n^i\wedge \alpha_n^i=0$ and $
\int_\Omega  \alpha_n^i\wedge \beta_n^k   =  (\alpha^i\wedge \beta^k) |\Omega|$ for all $i,k,$  by (\ref{approx-1}), we have 
\begin{equation}\label{approx-11} 
 \alpha^1\wedge \beta^1=\alpha^2\wedge \beta^2 =0,   \quad \beta^1\wedge \alpha^2= \alpha ^1\wedge \beta ^2=\lambda, 
 \end{equation}
 where 
 \[
\lambda=  \lim_{n\to \infty}\aint_\Omega   h(\alpha_n^1\wedge \alpha_n^2)(\alpha_n^1\wedge \alpha_n^2)\,dx.
  \]
  
{\bf  Case 1:}  Assume $\alpha ^1\wedge \beta ^2=\beta^1\wedge \alpha^2=\lambda=0.$ In this case, since  $h(t)t\ne 0$  for all $t\ne 0$, we have
\[
\lim_{n\to \infty}\int_\Omega   |h(\alpha_n^1\wedge \alpha_n^2)(\alpha_n^1\wedge \alpha_n^2)|\,dx=0.
\]
Consequently,  via possibly  a subsequence,  we have $h(\alpha_n^1\wedge \alpha_n^2)(\alpha_n^1\wedge \alpha_n^2)\to 0,$ and thus  $\alpha_n^1\wedge \alpha_n^2=\det Du_n \to 0,$ as $n\to\infty,$  a.e.\,on $\Omega.$  This implies  $ \det A=\aint_\Omega \det Du_n \to 0;$ thus $\det A=0.$  From  $h(0)=0$ and (\ref{approx}), by Jensen's inequality and the bounded convergence theorem, it follows that
\[
|B|\le  \lim_{n\to \infty}\aint_\Omega  |Dv_n| \,dx =0;
 \]
thus $B=0;$ in this case,  $B=\mu h(\det A)A=0$ for any $\mu>0.$ 

{\bf Case 2:}  Assume  $\alpha ^1\wedge \beta ^2=\beta^1\wedge \alpha^2=\lambda\ne 0.$ In this case, all $\alpha^i$ and $\beta^i$ are nonzero. Moreover, by (\ref{approx-11}), we have $\beta^1=t\alpha^1$ and $\beta^2=s\alpha^2$ for some constants $t,s\in\R;$  thus  $\lambda=\alpha ^1\wedge \beta ^2=\beta^1\wedge \alpha^2=t\alpha^1\wedge \alpha^2=s\alpha^1\wedge \alpha^2.$ Since $\lambda\ne 0,$ it follows that $t=s\ne 0$ and $\det A=\alpha^1\wedge \alpha^2\ne 0.$ Therefore, in this case, we also have $B=tA=\mu h(\det A) A,$ where 
\[
\mu = {\lambda}/{ [h(\det A)\det A}]>0,
\]
 as the function $th(t)$ remains the same sign as $\lambda\ne 0$ for all $t\ne 0.$ \end{proof}
 
  \subsection{A  rescaled Dirichlet problem}
Let  $\det A\ne 0$ and $B=\mu h(\det A)A$ be such that there exist  sequences $\{u_n\}$ and $\{v_n\}$ uniformly bounded in $W^{1,\infty}(\Omega;\R^2)$  and satisfying (\ref{approx}).   We introduce the rescaling:
\[
\tilde u_n=A^{-1}u_n,\quad \tilde v_n=h(\det A) A^{-1}v_n,\quad \tilde h(t)=\frac{h(t\det A)}{h(\det A)} 
\]
to obtain  $ \tilde u_n(x)|_{\partial\Omega}= x,\;  
 \tilde v_n(x)|_{\partial\Omega}=\mu x,$ and 
\[
 \lim_{n\to \infty} \int_\Omega |\tilde h(\det D\tilde u_n)  D\tilde u_n - D\tilde v_n|\,dx =0,
 \]
 where   $\tilde h(0)=0$, $\tilde h(1)=1$ and $\tilde h$ is {\em strictly increasing} on $\R.$
 
In what follows, we  drop  the tilde signs and study whether  for some positive $\mu\ne 1$ there exist  
sequences $\{ u_n\}$ and $\{ v_n\}$ uniformly bounded in $W^{1,\infty}(\Omega;\R^2)$  such that 
 \begin{equation}\label{approx-a}\begin{cases} 
 \displaystyle{\lim_{n\to \infty} \int_\Omega | h(\det D u_n)  D u_n - D v_n|\,dx =0,}\\[2ex]
  u_n(x)|_{\partial\Omega}= x,\;\;\; 
   v_n(x)|_{\partial\Omega}=\mu x,
\end{cases}
\end{equation}
where  $ h\in C(\R)$ is such that  
\begin{equation}\label{spec-h}
\mbox{$ h(0)=0, \; \; h(1)=1$ and $ h$ is {\em strictly increasing} on $\R.$}
\end{equation}

 \begin{pro}\label{lem43}  Suppose that  
sequences $\{ u_n\}$ and $\{ v_n\}$ are uniformly bounded in $W^{1,\infty}(\Omega;\R^2)$   satisfying (\ref{approx-a}) such that $\det Du_n\ge \epsilon_0$ a.e.\,on $\Omega,$ where $\epsilon_0>0$ is a constant.  Then $\mu =1.$\end{pro}

\begin{proof}\  From (\ref{approx-a}) we have
 \[
 \begin{split}  \lim_{n\to \infty}\int_\Omega  h(\det D u_n)(\det D u_n)\,dx =  \lim_{n\to \infty}\int_\Omega Du_n^1\wedge Dv_n^2\,dx=\mu|\Omega|, \\ 
 \lim_{n\to \infty}\int_\Omega  h^2(\det D u_n)(\det D u_n) \,dx=  \lim_{n\to \infty}\int_\Omega \det Dv_n\,dx =\mu^2 |\Omega|,
 \end{split}
 \]
and thus 
 \begin{equation}\label{approx-b}
 \lim_{n\to \infty} \int_\Omega  (h(\det D u_n)-\mu)^2  \det D u_n  \,dx  =0.
 \end{equation}
Since $\det Du_n\ge \epsilon_0$ a.e.\,on $\Omega,$  it follows that $h(\det D u_n)-\mu\to 0$ strongly  in $L^2(\Omega)$ as $n\to\infty.$ Since $h$ is one-to-one, we have $\det Du_n\to h^{-1}(\mu)$ a.e.\,on $\Omega$ as, perhaps 
via a subsequence, $n\to \infty.$ Finally, as  $\int_\Omega \det Du_n\,dx=|\Omega|,$ we have $h^{-1}(\mu)=1$ and thus $\mu=1.$
\end{proof}

Uniformly bounded  sequences $\{ u_n\}$ and $\{ v_n\}$ in $W^{1,\infty}(\Omega;\R^2)$  satisfying (\ref{approx-a}) can be constructed {\em provided} that   the  Dirichlet problem:   
\begin{equation}\label{pdr-1}\begin{cases}
 h(\det D u) Du=D v  \quad a.e. \; \Omega,\\
 u(x)|_{\partial\Omega}=x, \\
   v(x)|_{\partial\Omega}=\mu  x,
\end{cases}
\end{equation}
has a Lipschitz solution  $(u,v)\in W^{1,\infty}(\Omega;\R^4).$    

 \begin{remk} (i) If $\mu=1,$ then $(u,v)=(x,x)$ is a Lipschitz solution of problem (\ref{pdr-1}). If  (\ref{pdr-1}) has a Lipschitz solution $(u,v)$ for $\mu\ne 1$ then it will have {\em infinitely many} Lipschitz solutions.
 
 (ii) For some positive $\mu \ne 1$,  (\ref{pdr-1}) may have solutions $(u,v)\in C(\bar\Omega;\R^4)\cap W^{1,1}(\Omega;\R^4).$  
  For example (see also Example \ref{ex1}), let $\Omega$ be the unit open disk in $\R^2$ and $\mu >1$ and   $\lambda >1$ be  such that $h(\lambda)=\mu.$ Define the radial functions  $u(x)= \phi(|x|) x$ and $ v(x)=\mu u (x),$ where 
\[
\phi(r)=\begin{cases} 0 & \mbox{if $0\le r\le \sqrt{\frac{\lambda-1}{\lambda }},$}\\
\sqrt{\lambda -\frac{\lambda -1}{r^2}} & \mbox{if $ \sqrt{\frac{\lambda -1}{\lambda }}\le r\le 1.$}
\end{cases}
\]
Then $(u,v)\in C(\bar\Omega;\R^4)\cap W^{1,p}(\Omega;\R^4)$ for  all $1\le p<2$
is a solution of  (\ref{pdr-1}). However,   $(u,v)\notin  W^{1,p}(\Omega;\R^4)$ for any $p\ge 2.$
 \end{remk}

The following  result asserts that if  problem (\ref{pdr-1}) has a Lipschitz solution  for some positive $\mu\ne 1$ then $\det  Du$ must change signs  on $\Omega.$

 \begin{thm} \label{thm44}  Let  $\mu>0$ and $u,v\colon \bar\Omega\to\R^2$ be Lipschitz solutions of (\ref{pdr-1}) such that $\det Du \ge 0$   a.e.\,in $\Omega.$ Then $\det Du=1$ a.e.\,on $\Omega;$ thus, $\mu=1$ and $u= v$ on $\bar\Omega.$
 \end{thm}
 
 \begin{proof}\  Write $g=h(\det Du).$ Since  $gDu=Dv$, we have
 \[
 g^2\det Du=\det Dv,\quad 2g\det Du= Dv:\cof Du \quad a.e.\; \Omega.
 \]
 Integrating over $\Omega$  and using  the given boundary conditions and the fact that $\det A$ and $B:\cof A$ are null-Lagrangians, we have that
 \[
\aint_\Omega \det Du=1,\quad  \aint_\Omega g^2 \det Du =\mu^2, \quad  \aint_\Omega g \det Du  =\mu,
 \]
and thus 
 \[
\aint_\Omega (g-\mu)^2\det Du = \aint_\Omega \Big  ( g^2 \det Du -2\mu g\det Du +\mu^2  \det Du\Big )=0.
\]
Since $\det Du\ge 0$ a.e.\,in $\Omega,$ it follows that  $g =h(\det Du)=\mu$ a.e.\,on the set $E=\{x\in\Omega:\det Du(x)> 0\}.$ As $h$ is one-to-one, we have $\det Du=\lambda \chi_E$, where $h(\lambda)=\mu.$ Therefore, $\det Dv=\mu^2\lambda \chi_E.$ Since $v$ is Lipschitz, we assume  $|Dv|\le M$ for some $M>0;$  let $L=\frac{M^2}{\mu^2\lambda}.$ Then it is easily seen that
\[
|Dv(x)|^2\le L\det Dv(x) \quad a.e.\;\Omega.
\]
Consequently,  $v$ is a non-constant  {\em $L$-quasiregular map} on $\Omega.$ It is well-known  that a non-constant $L$-quasiregular mapping cannot have zero Jacobian determinant on a set of positive measure  (see \cite{LV,Re}); thus, it follows that $|\Omega\setminus E|=0.$ This proves   $\det Du=\lambda$ a.e.\,on $\Omega$, from which we have
$\lambda=\aint_\Omega \det Du =1;$ hence $\mu=1$ and $u=v$ on $\bar\Omega.$
 \end{proof}

\end{document}